\documentclass[12pt]{amsart}

\usepackage{amsmath,amsthm,amssymb}
\usepackage[dvips]{graphicx}
\usepackage{daytime}

\newtheorem{thm}{Theorem}[section]
\newtheorem{dummy}{Theorem}
\newtheorem{dummyy}{Corollary}
\newtheorem{dummyyy}{Conjecture}
\newtheorem{prop}[thm]{Proposition}
\newtheorem{prob}[thm]{Problem}
\newtheorem{lem}[thm]{Lemma}
\newtheorem{conj}[thm]{Conjecture}
\newtheorem{cor}[thm]{Corollary}
\newtheorem{clm}[thm]{Claim}

\theoremstyle{definition}
\newtheorem{ex}[thm]{Example}

\theoremstyle{remark}
\newtheorem{rem}[thm]{Remark}

\newcommand{\R}{\mathbb{R}}
\newcommand{\N}{\mathbb{N}}
\newcommand{\Z}{\mathbb{Z}}

\numberwithin{equation}{section}

\title{Annulus twist and diffeomorphic 4-manifolds}

\author{Tetsuya Abe} 
\address{Research Institute for Mathematical Sciences, 
Kyoto University, Kyoto 606-8502, Japan} 
\email{tetsuya@kurims.kyoto-u.ac.jp}
\thanks{The first author was partially supported by KAKENHI,
Grant-in-Aid for Research Activity start-up (No. 00614009),
Japan Society for the Promotion of Science.}

\author{In Dae Jong}
\address{Faculty of Liberal Arts and Sciences, Osaka Prefecture University, 
1-1 Gakuen-cho, Nakaku, Sakai, Osaka 599-8531, Japan} 
\email{jong@las.osakafu-u.ac.jp}

\author{Yuka Omae} 
\address{Osaka Prefectural Kitano High School,
Osaka 532-0025,
Japan} 
\email{T-OmaeYu@medu.pref.osaka.jp}

\author{Masanori Takeuchi} 
\address{Department of Mathematics, Graduate School of Science, Osaka University, 
Toyonaka, Osaka 560-0043, Japan}

\keywords{Annulus twist; Band presentation; Carving; Dehn surgery; Kirby calculus; Homotopy $4$-ball; Slice knot; Unknotting number one knot}

\begin{document}

\maketitle

\begin{abstract}
We give a method for obtaining infinitely many framed knots which represent a diffeomorphic $4$-manifold. 
We also study a relationship between the $n$-shake genus and the 4-ball genus of a knot.
Furthermore we give a construction of homotopy $4$-spheres from a slice knot with unknotting number one.
\end{abstract}

\section{Introduction}\label{sec:intro}
In Kirby's problem list~\cite{Kirby}, Clark asked the following. 
\begin{prob}[{\cite[Problem 3.6 (D)]{Kirby}}]\label{prob:3}
Is there a $3$-manifold which can be obtained by $n$-surgery on infinitely many mutually distinct knots? 
\end{prob}
There are many related studies on Problem~\ref{prob:3} (e.g., see \cite{Brakes}, \cite{Gompf2}, \cite{Lickorish2}, \cite{Livingston}, \cite{Teragaito3}). 
In 2006, Osoinach \cite{Osoinach} proved that 
there exists a hyperbolic $3$-manifold which is obtained by $0$-surgery on infinitely many mutually distinct knots. 
He also constructed such a toroidal $3$-manifold. 
In particular, he solved Problem~\ref{prob:3} for $n=0$.
The key tool to produce infinitely many (mutually distinct) knots was an operation which we will call an annulus twist.
By the same method, Teragaito~\cite{Teragaito} proved that 
there exists a Seifert fibered $3$-manifold which is obtained by $4$-surgery on infinitely many mutually distinct knots.
In particular, he solved Problem~\ref{prob:3} for $n=4$.　
It is known that Problem~\ref{prob:3} is true for $n = 4m$ $(m \in \Z)$ by Kouno~\cite{Kouno}. 
For a recent study, see \cite{Teragaito2}.

For a knot $K$ in the 3-sphere $S^3= \partial D^4$ and $n \in \Z$, 
we denote by $M_{K}(n)$ the $3$-manifold obtained by $n$-surgery on $K$
and by $X_{K}(n)$ the smooth $4$-manifold obtained from the 4-ball $D^4$
by attaching a $2$-handle along $K$ with framing $n$. 
Note that $\partial X_{K}(n) \approx M_{K}(n)$,
where we use the notation ``$X \approx Y$'' when two manifolds $X$ and $Y$ are diffeomorphic. 
As a 4-dimensional analogue of Problem~\ref{prob:3}, the following is natural to ask. 

\begin{prob}\label{prob:4}
Let $n$ be an integer. 
Find infinitely many mutually distinct knots $K_1$, $K_2, \cdots$ 
such that $X_{K_i}(n) \approx X_{K_j}(n) $ for each $i, j \in \N$. 
\end{prob}

Currently less is known about Problem~\ref{prob:4}. 
We review known results on Problem~\ref{prob:4}. 
Due to Akbulut~\cite{Akbulut2, Akbulut1}, 
there exists a pair of distinct knots $K_n$ and $K_n'$ such that 
$X_{K_{n}}(n) \approx X_{K'_{n}}(n)$ for each $n \in \Z$. 
After Akbulut's works, the fourth author~\cite{Takeuchi1} solved Problem~\ref{prob:4} for $n=4$.
Indeed, let $K_1$, $K_2, \cdots$ be the knots
constructed by Teragaito in \cite{Teragaito}.
Then  $\partial X_{K_i}(4) \approx \partial X_{K_j}(4) $ for each $i, j \in \N$.
The fourth author proved that 
a diffeomorphism $f : \partial X_{K_i}(4) \to \partial X_{K_j}(4)$
extends to a diffeomorphism $\tilde{f} : X_{K_i}(4) \to X_{K_j}(4)$ such that 
$\tilde{f} |_{\partial X_{K_i}(4)}=f$ 
using carving techniques (Lemma~\ref{lem:extend}) introduced by Akbulut in \cite{Akbulut1}. 
Subsequently, the third author~\cite{Omae1} solved Problem~\ref{prob:4} for $n=0$.
%After Akbulut's works, 
%the fourth author \cite{Takeuchi1} developed the study on Problem~\ref{prob:4}. 
%He gave infinitely many  knots $K_1$, $K_2, \cdots$ 
%such that $X_{K_i}(4) \approx X_{K_j}(4) $ for each $i, j \in \N$, solving Problem~\ref{prob:4} for $n=4$. 
%Subsequently the third author  \cite{Omae1} studied Problem~\ref{prob:4} for $n=0$. 
%She gave infinitely many  knots $K_1$, $K_2, \cdots$ 
%such that $X_{K_i}(0) \approx X_{K_j}(0) $ for each $i, j \in \N$, 
%however, did not prove that the knots are mutually distinct. 

In this paper, 
we generalize these results and give a framework to solve Problem~\ref{prob:4} for $n=0$ or $\pm 4$. 
To state our main theorem,
we introduce a band presentation of a knot as follows:
Let $A \subset \R^2 \cup \{ \infty \} \subset S^3$ be a standardly embedded annulus
with a $\varepsilon$-framed unknot $c$ in  $S^3$ as shown in the left side of Figure~\ref{fig:Def-BP}, where $\varepsilon = \pm 1$. 
Take an embedding of a band $b$: $I \times I \to S^3$ such that 
\begin{itemize}
\item $b(I \times I) \cap \partial A = b(\partial I \times I)$, 
\item $b(I \times I) \cap \text{int} A$ consists of ribbon singularities, and
\item $b(I \times I)  \cap c= \emptyset$,
\end{itemize}
where $I = [0,1]$. 
If a knot $K$ is equivalent to the knot $\left( \partial A \setminus b(\partial I \times I)\right) \cup b( I \times \partial I)$
in $M_{c}(\varepsilon)$, 
then we say that $K$ admits a \textit{band presentation} $(A,b,c,\varepsilon)$\footnote{Winter~\cite{Winter} also introduced the notion of a band presentation in the study of ribbon 2-knots. 
A band presentation introduced in the current paper is different from Winter's one. }. 
Note that $M_{c}(\varepsilon) \approx S^3$. 
A typical example of a band presentation of a knot is given in Figure~\ref{fig:Def-BP}.

\begin{figure}[htb]
\includegraphics[width=1.0\textwidth]{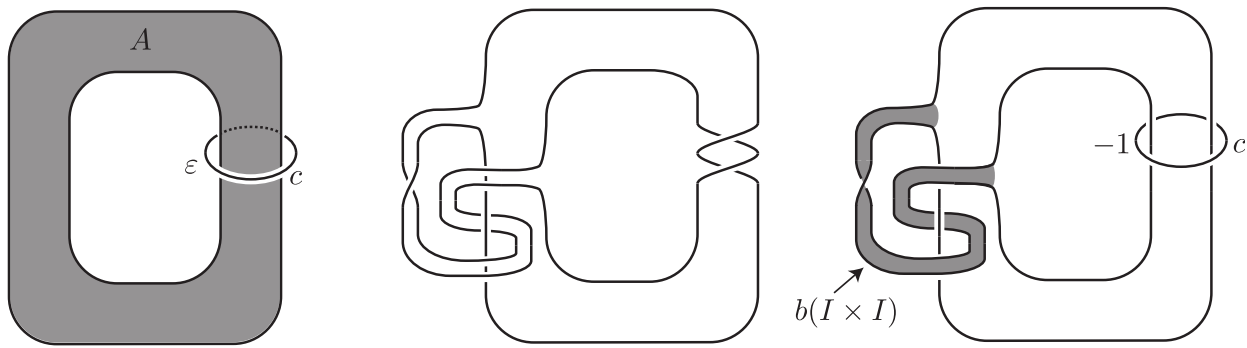}
\caption{The knot depicted in the center admits a band presentation as shown in the right side.}
\label{fig:Def-BP}
\end{figure}
\begin{figure}[htb]
\includegraphics[width=.1\textwidth]{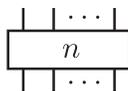}
\caption{$n$ times right (resp.~left)-handed full twists for $n>0$ (resp.~$n<0$). }
\label{fig:FullTwists}
\end{figure} 
\begin{figure}[htb]
\includegraphics[width=1.0\textwidth]{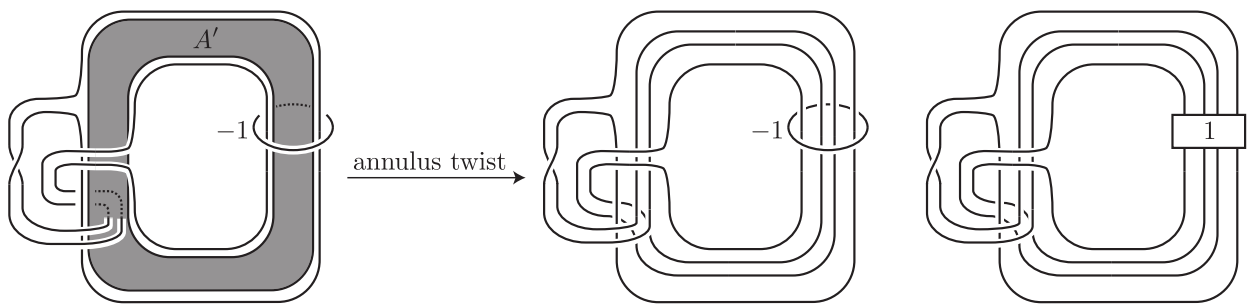}
\caption{The associated annulus $A'$, an annulus twist, and the resulting knot.}
\label{fig:Ex-AT}
\end{figure} 

We explain how to obtain  infinitely many knots  from a knot admitting  a band presentation.
We represent $n$ times full twists by the rectangle labelled $n$ as shown in Figure~\ref{fig:FullTwists}.
For any knot $K$ admitting a band presentation, 
we can find an annulus $A'$ and twist the knot along $A'$ as shown in Figure~\ref{fig:Ex-AT}. 
The resulting knot is called the knot obtained from $K$ by applying  an annulus twist.
Similarly, we define the knot obtained from $K$ by applying  an annulus twist $n$ times. 
For precisely, see Subsection~\ref{subsec:AnnulusTwist}. 
Note that, for a knot admitting a band presentation, we can associate a framing  (called the induced framing)
in a natural way (see Subsection~\ref{subsec:Band}).
Our main result is the following.

\renewcommand{\thedummy}{\ref{thm:main}}
\begin{dummy}  
Let $K$ be a knot admitting a band presentation, 
$K_{n}$ the knot obtained from $K$ by applying an annulus twist $n$ times for $n \in \Z$.
Then $X_{K}(\gamma) \approx X_{K_{n}}(\gamma)$,
 where $\gamma$ is the induced framing from the band presentation. 
\end{dummy}
The key to prove Theorem~\ref{thm:main} is also carving techniques introduced in \cite{Akbulut1}.
%We briefly explain the proof of Theorem~\ref{thm:main}. 
%Let $K$ be a knot with a band presentation,
%$K'$ the knot obtained from $K$ by applying an annulus twist, 
%and $\gamma$ the induced framing from the band presentation. 
%First, we give a  diffeomorphism  $f: \partial X_{K}(\gamma) \to \partial X_{K'}(\gamma)$
%by Kirby calculus (as explained in \cite{Teragaito}). 
%The problem is whether $f$ extends to a diffeomorphism $\tilde{f}: X_{K}(\gamma) \to X_{K'}(\gamma)$
%such that $\tilde{f}|_{ \partial X_{K}(\gamma)}=f$. 
%Akbulut \cite{Akbulut1} introduce a sufficient condition for $f$ to be  extended to $\tilde{f}$ 
%(Lemma~\ref{lem:extend}) using Cerf's theorem which 
%states that any diffeomorphism of  $S^3$ extends over $D^4$ (e.g.~see \cite{Cerf}). 
%We can check that $f$ satisfies the condition, proving Theorem~\ref{thm:main}.
We note that Theorem~\ref{thm:main} gives a general answer to \cite[Excerise 15]{AkbulutBook}.
As a corollary of Theorem~\ref{thm:main}, we obtain the following.

\renewcommand{\thedummyy}{\ref{cor:problem}}
\begin{dummyy} 
For each $\gamma \in \{0, \pm 4\}$, there exist infinitely many mutually distinct knots $K_1, K_2, \cdots$
such that  $X_{K_{i}}(\gamma) \approx X_{K_{j}}(\gamma)$ for each $i, j \in \N$. 
\end{dummyy}

It is important to ask which knot admits a band presentation.
The set of knots admitting band presentations contains all unknotting number one knots (Lemma~\ref{lem:band-presentation}). 
Thus the set is relatively large in this sense. 
Furthermore this observation enables us to construct homotopically slice knots and homotopy 4-spheres from a slice knot with unknotting number one (Proposition~\ref{prop:un2}).

\bigskip

In studies on smooth structures of a $4$-manifold, 
Akbulut \cite{Akbulut1} introduced the {\it $n$-shake genus} $g_{s}^{n} (K)$ of a knot $K$, 
which is defined as the minimal genera of closed smooth surfaces in $X_{K}(n)$
such that the homology class of each closed surface generates $H_2(X_{K}(n))$. 

In this paper, we also study a relationship between the $n$-shake genus $g_{s}^{n} (K)$ and the $4$-ball genus $g_{*}(K)$ of a knot $K$.
It is easy to see that $g_{s}^{n} (K) \le g_{*}(K)$. 
Akbulut and Kirby \cite[Problem~1.41 (A)]{Kirby} asked whether $g_s^0(K) = g_{*}(K)$, 
and noted that these probably do not coincide. 
One of our motivations is to find a knot $K$ with $g_{s}^{0} (K) < g_{*}(K)$. 
The following is an immediate consequence of  Theorem~\ref{thm:main}.

\renewcommand{\thedummyy}{\ref{cor:ShakeGenus}}
\begin{dummyy} 
Let $K$ be a knot with a band presentation whose induced framing is $0$, 
$K_{n}$ the knot obtained from $K$ by applying an annulus twist $n$ times. 
Then, for each $i, j \in \Z$, 
$$g_{s}^0(K_{i})=g_{s}^0({K_{j}}) \, .$$ 
\end{dummyy}

With the notation of Corollary~\ref{cor:ShakeGenus}, 
if  $g_{*}(K) \ne g_{*}(K_{n})$ for some $n \in \Z$,
then we see that $g_{s}^{0} (K_n) < g_{*}(K_n)$.
However, according to some observations we would propose the following.

\renewcommand{\thedummyyy}{\ref{conj:AJ}}
\begin{dummyyy} 
Let $K$ be a knot with a band presentation whose induced framing is $0$, 
and $K_{n}$ the knot obtained from $K$ by applying an annulus twists $n$ times. 
Then $g_{*}(K)=g_{*}({K_{n}})$.  
\end{dummyyy}

We note that the situation is entirely different for the $n$-shake genus with $n \ne 0$.
In fact, Akbulut~\cite{Akbulut2, Akbulut1} and the third author~\cite{Omae1}  gave a knot $K$ such that $g_{s}^{n} (K) < g_{*}(K)$ for each $n \ne 0$ (see also \cite{Lickorish}). 
%The third author \cite{Omae1} also constructed  such a knot. 
We give infinitely many such knots as follows.

\renewcommand{\thedummyy}{\ref{cor:nShakeGenus}}
\begin{dummyy}
For each integer  $n \ne 0$, 
there exist infinitely many knots $K_1, K_2, \dots$ such that 
$ g_{s}^{n}(K_i) < g_{*}(K_i)$ for any $i \in \N$. 
\end{dummyy}

\subsection*{Acknowledgments}
The authors would like to express their gratitude to Hisaaki Endo, 
Motoo Tange, and Kouichi Yasui for teaching them about $4$-manifold theory. 
They also would like to thank Masakazu Teragaito for his useful comments, 
and Makoto Sakuma for suggesting the fact in Remark~\ref{rem:Sakuma}.
They also thank the referee for careful reading of our draft and useful suggestions.

\section{Construction of diffeomorphic $4$-manifolds}\label{sec:MainThm}

In this section, we give a method for obtaining a family of framed knots 
such that each represents the same $4$-manifold up to diffeomorphism. 
%For  details on this topic, see Gompf and Stipsicz's book \cite{Gompf}.

\subsection{Band presentation}\label{subsec:Band} 

We study some properties of a band presentation. 
Let $K$ be a knot admitting a band presentation $(A, b, c, \varepsilon)$.
Note that the boundary of the (immersed) surface $A \cup b(I \times I)$ is $K$ in $M_{c}(\varepsilon) \approx S^3$. 
Thus we can regard $K$ as a framed knot by using the surface $A \cup b(I \times I)$
in $M_{c}(\varepsilon) \approx S^3$ and call it the {\it induced framing} from the band presentation. 
The induced framing from the band presentation is $0$ if $A \cup b(I \times I)$ is orientable, 
$\pm 4$ if $A \cup b(I \times I)$ is non-orientable and $\varepsilon = \mp 1$.
%Unless otherwise noted, 
%we regard a knot with a band presentation as a framed knot with framing. 
Here we give examples of band presentations.

\begin{ex}\label{ex:Jn}
Let $J_{m}$ $(m \in \Z)$ be the knot depicted in the upper half of Figure~\ref{fig:Jm}. 
Then we see that $J_m$ admits a band presentation as shown in the left side of the lower half of Figure~\ref{fig:Jm}.  
The induced framing from this band presentation is $0$ if $m$ is odd, otherwise it is $4$. 
Note that a band presentation of a given knot may not be unique. 
For example, $J_{m}$ also admits a  band presentation as shown in the right side of the lower half of Figure~\ref{fig:Jm}. 
In this case, the induced framing is $0$ if $m-1$ is odd, otherwise it is $-4$. 
Teragaito~\cite{Teragaito} and the fourth author~\cite{Takeuchi1} studied the knot $J_{2}$ and 
 the third author~\cite{Omae1} studied the knot $J_{1}$ respectively. 
\end{ex}

\begin{figure}[htb]
\includegraphics[width=.95\textwidth]{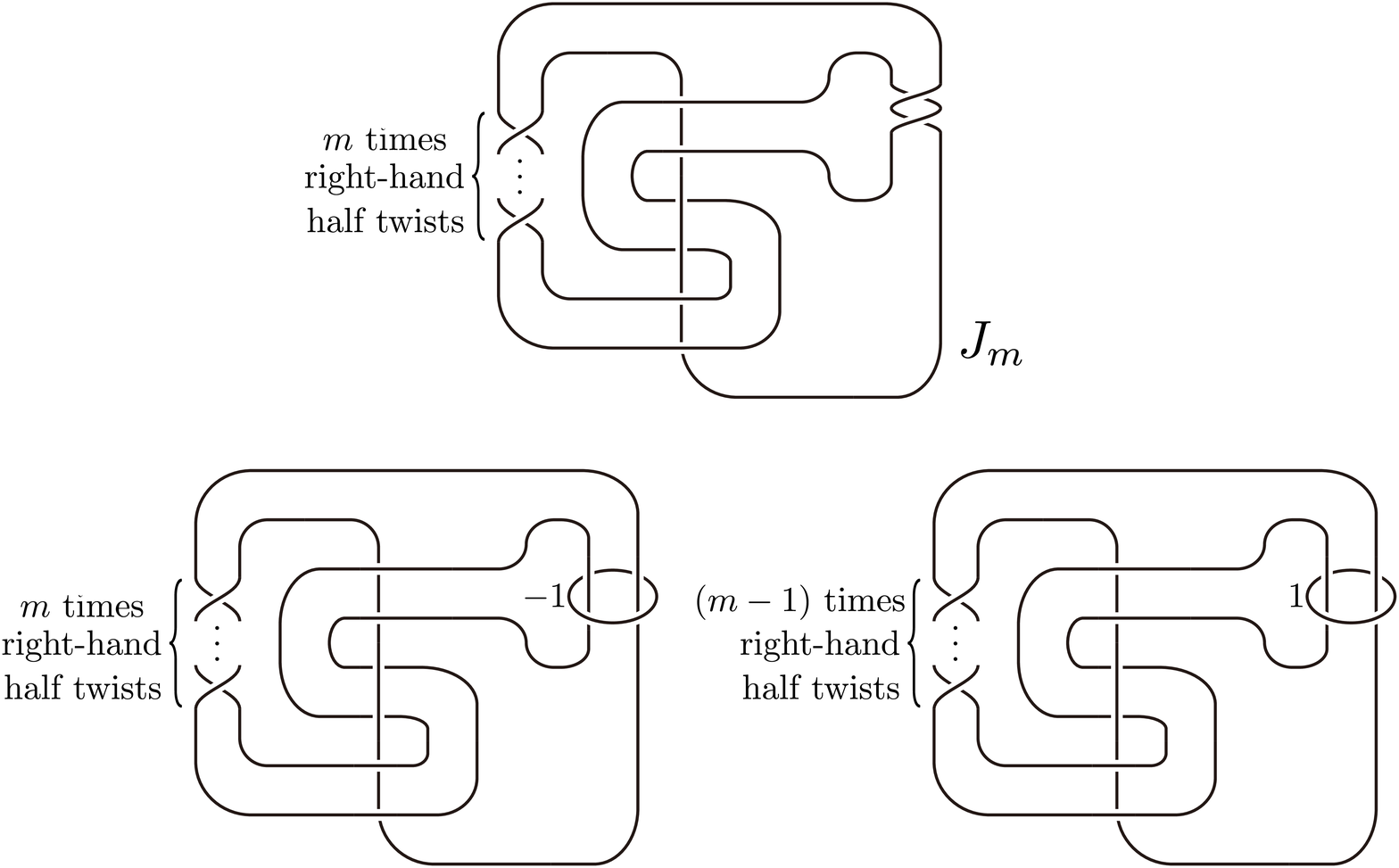}
\caption{The knot $J_m$ and its band presentations.}
\label{fig:Jm}
\end{figure}

It is easy to see that the unknotting number of $J_{m}$ is one. 
This fact is generalized as follows.

\begin{lem}\label{lem:band-presentation}
Let $K$ be a knot. 
If the unknotting number of $K$ is less than or equal to one, 
then  $K$ admits a band presentation. 
\end{lem}
\begin{proof}
By the assumption, there exists a crossing of a diagram of $K$ such that changing the crossing yields the unknot. 
Adding a band near the crossing as shown in the left side of Figure~\ref{fig:unknotting-band}, we obtain a Hopf link which bounds a (twisted) annulus. 
Then we see that $K$ admits a band presentation.
\end{proof}

\begin{rem}
We  note that the set of knots admitting a band presentation is equal to the set of knots obtained from the Hopf link applying a band surgery. 
For studies on a band surgery, see~\cite{Abe1, Abe2} for example.
\end{rem}

\begin{figure}[htb]
\includegraphics[width=.55\textwidth]{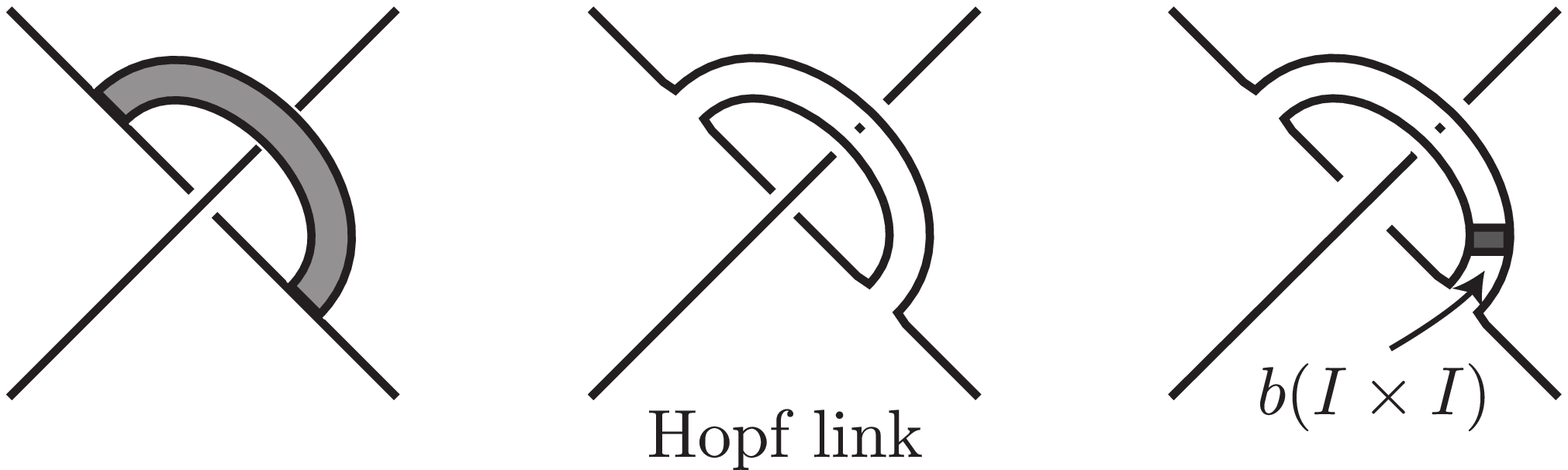}
\caption{}
\label{fig:unknotting-band}
\end{figure}

\subsection{Annulus twist} \label{subsec:AnnulusTwist} 
Using an annulus twist,
Osoinach~\cite{Osoinach} constructed infinitely many knots such that each knot yields the same $3$-manifold by $0$-surgery. 
Let $A$ be a standardly embedded annulus in $S^3$ and set $\partial A = c_1 \cup c_2$. 
% as shown in Figure~\ref{fig:annulus}. 
An \textit{annulus twist along $A$} is to apply Dehn surgery on $c_1$ and $c_2$ along slopes $1$ and $-1$ respectively, which yields a diffeomorphism $\varphi : S^3 \setminus \nu \partial A \to S^3 \setminus \nu \partial A$, 
where $\nu \partial A$ is an open tubular neighborhood of $\partial A$ in $S^3$. 
%The surgery is described as a diffeomorphism $f : S^3 \to S^3$ such that 
%\begin{align*}
%& f|_{A \times [0,1]} (((e^{i\theta}, s), t)) = ((e^{i(\theta+2\pi t)},s), t)  \, , \\
%& f|_{S^3 \setminus (A \times [0,1])} = id \, ,
%\end{align*}
%where $A$ is parametrized as  $S^1 \times [0,1] = \left\{ (e^{i\theta},s) \, | \, 0 \leq \theta \leq 2\pi \, , 0 \leq s \leq 1 \right\}$. 
In the case where the slope of $c_1$ is $-1$ and that of $c_2$ is $1$, 
it is described as the diffeomorphism $\varphi ^{-1}$.
%These diffeomorphisms are described as twistings along $A$ as shown in Figure~\ref{fig:twisting1}. 
In general, the surgery on $c_1$ and $c_2$ along slopes $1/n$ and $-1/n$ for $n \in \Z$ is 
described as the diffeomorphism $\varphi^n$, that is, twisting $n$ times.

%\begin{figure}[htb]
%\includegraphics[scale=.3]{annulus.eps}
%\caption{}\label{fig:annulus}
%\end{figure}
%\begin{figure}[htb]
%\includegraphics[scale=.30]{AT.eps}
%\caption{The diffeomorphisms $f$ and $f^{-1}$.}\label{fig:twisting1}
%\end{figure}

As mentioned in Section~\ref{sec:intro}, for a given knot with a band presentation, 
we can construct infinitely many knots. 
Here we explain this process precisely. 
Let $K$ be a knot admitting a band presentation $(A,b,c,\varepsilon)$. 
Shrinking the annulus $A$ slightly, we obtain an annulus $A' \subset A$ as shown in Figure~\ref{fig:Ex-AT}. 
We call $A'$ the {\it associated annulus} of $A$. 
Applying an annulus twist $n$ times along $A'$, we obtain a knot $K_n \subset S^3$ with a band presentation $(A, b_n, c, \varepsilon)$, where $b_{n}$ is the obvious one. 
Note that the induced framing of $(A, b_n, c, \varepsilon)$ coincides with that of $(A, b, c, \varepsilon)$. 
We call $K_n$ the knot obtained from $K$ by applying an annulus twist $n$ times without mentioning the associated annulus $A'$. 
%Note that $K_{n}$ has a band presentation.  
%Let $\gamma$ be the induced framing  $\gamma \in \left\{ 0, \pm4 \right\}$ of $K$. 
%Then  the induced framing of $K_{n}$
%is also $\gamma$. 

\begin{rem}\label{rem:n-times}
$K_{n+m}=(K_{n})_{m}$ for $n, m \in \Z$. 
\end{rem}

The following proposition follows from \cite[Theorem 2.3]{Osoinach}.

% Note that knots $K_n$ might be the same knots.
%By the properties of annulus twists introduced in \cite[Theorem 2.3]{Osoinach}, we have the following. 
\begin{prop}\label{prop:3-diffeo} 
Let $K$ be a knot with a band presentation, 
$K_{n}$ the knot obtained from $K$ by applying an annulus twist $n$ times.
Then $M_{K}(\gamma) \approx M_{K_{n}}(\gamma)$,
where  $\gamma$ is the induced framing from the band presentation. 
\end{prop}

\begin{rem}
Let $K_1$, $K_2, \cdots$ be knots 
such that $M_{K_{i}}(0) \approx M_{K_{j}}(0)$ for each $i, j \in \N$.
Then the Alexander modules of  $K_{i}$ and $K_{j}$  are isomorphic.
In particular, we have $\Delta_{K_{i}}(t) = \Delta_{K_{j}}(t)$ and $\sigma(K_{i}) =\sigma(K_{j})$. 
Here $\Delta_K(t)$ denotes the Alexander polynomial and $\sigma(K)$ denotes the signature of a knot $K$. 
Therefore it is not easy to distinguish $K_i$ and $K_j$. 
\end{rem}

\begin{rem}
In the definition of a band presentation, the annulus $A$ is not necessary to be unknotted. 
In fact, the argument in this subsection holds for a band presentation with a knotted embedded annulus in $S^3$. 
\end{rem}

\subsection{Main theorem} \label{subsec:TheConstructionOfKnots}

In this subsection, we prove the main theorem.

\begin{thm} \label{thm:main}
Let $K$ be a knot admitting a band presentation, 
$K_{n}$ the knot obtained from $K$ by applying an annulus twist $n$ times for $n \in \Z$.
Then $X_{K}(\gamma) \approx X_{K_{n}}(\gamma)$,
 where $\gamma$ is the induced framing from the band presentation. 
\end{thm}

We briefly explain the proof of Theorem~\ref{thm:main}. 
Let $K$ be a knot with a band presentation,
$K'$ the knot obtained from $K$ by applying an annulus twist, 
and $\gamma$ the induced framing from the band presentation. 
First, we give a  diffeomorphism  $f: \partial X_{K}(\gamma) \to \partial X_{K'}(\gamma)$
by Kirby calculus (as explained in \cite{Teragaito}). 
The problem is whether $f$ extends to a diffeomorphism $\tilde{f}: X_{K}(\gamma) \to X_{K'}(\gamma)$
such that $\tilde{f}|_{ \partial X_{K}(\gamma)}=f$. 
Akbulut \cite{Akbulut1} introduced a sufficient condition for $f$ to be extended to $\tilde{f}$ (Lemma~\ref{lem:extend}) using Cerf's theorem which states that any self-diffeomorphism of $S^3$ extends over $D^4$ (e.g.~see \cite{Cerf}). 
(A general sufficient condition is shown in \cite{AkbulutBook}.) 
We can check that $f$ satisfies the condition, proving Theorem~\ref{thm:main}.
Here we recall this sufficient condition and give a proof for the sake of the readers.

\begin{lem}[\cite{Akbulut1}]\label{lem:extend}
Let $K$ and $K'$ be knots in $S^3 = \partial D^4$ with a diffeomorphism $g : \partial X_{K}(n) \to \partial X_{K'}(n)$, 
and let $\mu$ be a meridian of $K$. 
Suppose that if $\mu$ is $0$-framed, then $g(\mu)$ is the $0$-framed unknot in the Kirby diagram representing $X_{K'}(n)$. 
If the Kirby diagram which consists of the $2$-handle represented by $K'$ with framing $n$ and the $1$-handle represented by  (dotted) $g(\mu)$ represents $D^4$, 
then $g$ extends to a diffeomorphism $\widetilde{g}: X_{K}(n) \to X_{K'}(n)$ 
such that $\tilde{g}|_{\partial X_{K}(n)}=g$. 
\end{lem}
\begin{proof}
Since $\mu$ and $g(\mu)$ are unknotted, these bound obvious properly embedded disks in $D^4$, 
$D_1$ and $D_2$ respectively.
Let $\nu \mu$ be an open tubular neighborhood of $\mu$ in $S^3 = \partial D^4$. 
Let $\nu D_i$ be an open tubular neighborhood of $D_i$ in $D^4$. 
Since the framing of $g(\mu)$ is $0$ if we assume that $\mu$ is $0$-framed, 
$g|_{\nu\mu}$ extends to a diffeomorphism $g' : \overline{\nu D_1} \to \overline{\nu D_2}$. 
Now $X_{K}(n) \setminus \nu D_1 $ is diffeomorphic to $D^4$ 
and by the assumption $X_{K'}(n) \setminus \nu D_2$ is also diffeomorphic to $D^4$. 
Therefore $g$ extends to a diffeomorphism $\widetilde{g}: X_{K}(n) \to X_{K'}(n)$
such that $\tilde{g}|_{ \partial X_{K}(n)}=g$.
\end{proof}

%\begin{rem}\label{rem:Akbulut}
%In Lemma~\ref{lem:extend}, the condition that $g(\mu)$ is the $0$-framed unknot can be extended to the condition that $g(\mu)$ is the $0$-framed slice knot. 
%See hogehoge for more details. 
%\end{rem}

Now we start the proof of Theorem~\ref{thm:main}. 

\begin{proof}[Proof of Theorem~\ref{thm:main}]
Let $(A,b,c,\varepsilon)$ be a band presentation of $K$ and $\gamma$ the induced framing. 
The proof is divided into four cases.

\medskip 
\noindent
{\bf Case 1.} $\varepsilon =1$ and $\gamma=0$.

\medskip

First, we consider the case where $K=J_{2}$ with the band presentation $(A,b,c,1)$ depicted in the right side of Figure~\ref{fig:Jm}.
Then the induced framing $\gamma$ is $0$. 
Figure~\ref{fig:Kirby-Moves1} illustrates a diffeomorphism from $\partial X_{K_{n+1}}(0)$ to $\partial X_{K_{n}}(0)$ for an integer $n$. 
To see this, recall that $K_{n+1}=(K_{n})_1$ as noted in Remark~\ref{rem:n-times}, and ignore the meridian $\mu$ of $K_{n+1}$  for a while, which is illustrated by a broken circle in Figure~\ref{fig:Kirby-Moves1}. 
The first Kirby diagram in Figure~\ref{fig:Kirby-Moves1}  is obtained from $X_{K_{n+1}}(0)$ by a blow-up. 
%\footnote{By the abuse of the notation,  we denote the knot $K_{n+1}$ with framing $0$ by $X_{K_{n+1}}(0)$}
Sliding the $-1$-framed unknot to the $0$-framed $K_{n}$, we obtain the second Kirby diagram.
Applying an annulus twist, we obtain the third Kirby diagram. 
After an ambient isotopy and a blow-down, we obtain the $0$-framed $K_{n}$ which represents $X_{K_{n}}(0)$. 
This Kirby calculus induces a diffeomorphism from $\partial X_{K_{n+1}}(0)$ to $\partial X_{K_{n}}(0)$ 
and we denote it by $f_{n+1}$.

Remember the meridian $\mu$ of $K_{n+1}$. 
We can check that $f_{n+1}(\mu)$ is the $0$-framed unknot in the Kirby diagram of $X_{K_n}(0)$ if $\mu$ is $0$-framed. 
Let $W$ be the $4$-manifold $D^4 \cup h^1 \cup h^2$, 
where $h^1$ is the $1$-handle represented by $f_{n+1}(\mu)$ with a dot and $h^2$ 
is the $2$-handle represented by $K_{n}$ with framing $0$.
Sliding $h^2$ over $h^1$, we obtain a canceling pair (see Figure~\ref{fig:4-manifold1}), implying that $W \approx D^4$.
By Lemma~\ref{lem:extend}, $X_{K_{n+1}}(0) \approx X_{K_{n}}(0)$,
proving Theorem~\ref{thm:main} in the case where $K=J_{2}$ with the band presentation.

Furthermore we see that the above argument can be applied to any knot admitting a band presentation with $\varepsilon = 1$ and $\gamma = 0$ since all of a blow-up, a blow-down, handle slides, and isotopies in Figures~\ref{fig:Kirby-Moves1} and~\ref{fig:4-manifold1} were done in a neighborhood of the annulus $A$. 
Now we have completed the proof for Case 1.

\medskip
\noindent
\textbf{Case 2.} $\varepsilon =1$  and $\gamma=-4$.

\medskip 

In this case, we can prove similarly as Case 1 since the induced framing $\gamma =-4$ changed to $0$ when we apply a blow-up as shown in Figure~\ref{fig:blow-up}.

\begin{figure}[htb]
\includegraphics[width=1.0\textwidth]{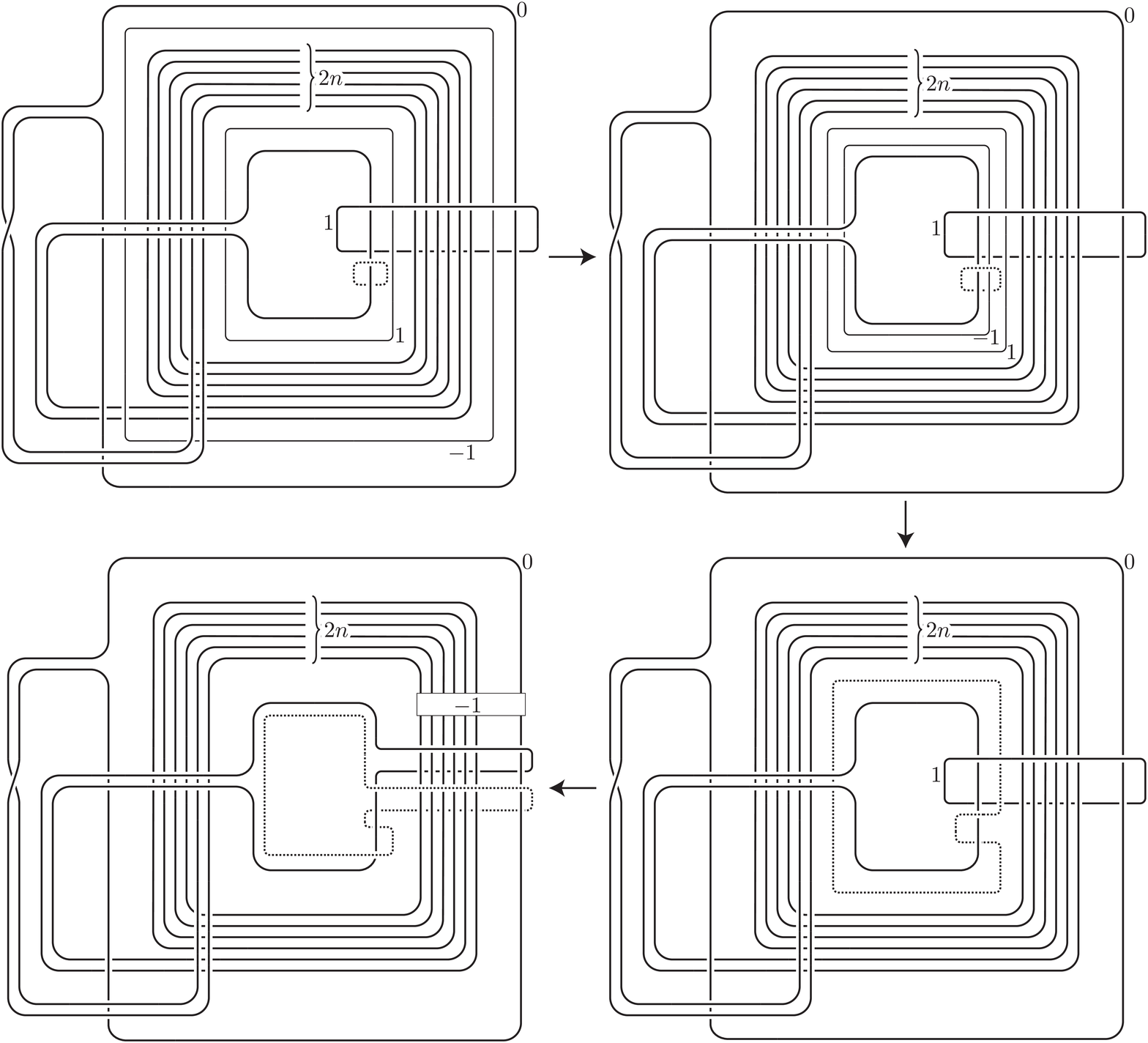}
\caption{Kirby diagrams representing a diffeomorphism $f_{n+1} : \partial X_{K_{n+1}}(0) \to \partial X_{K_{n}}(0) $.}
\label{fig:Kirby-Moves1}
\end{figure}
\begin{figure}[htb]
\includegraphics[width=1.0\textwidth]{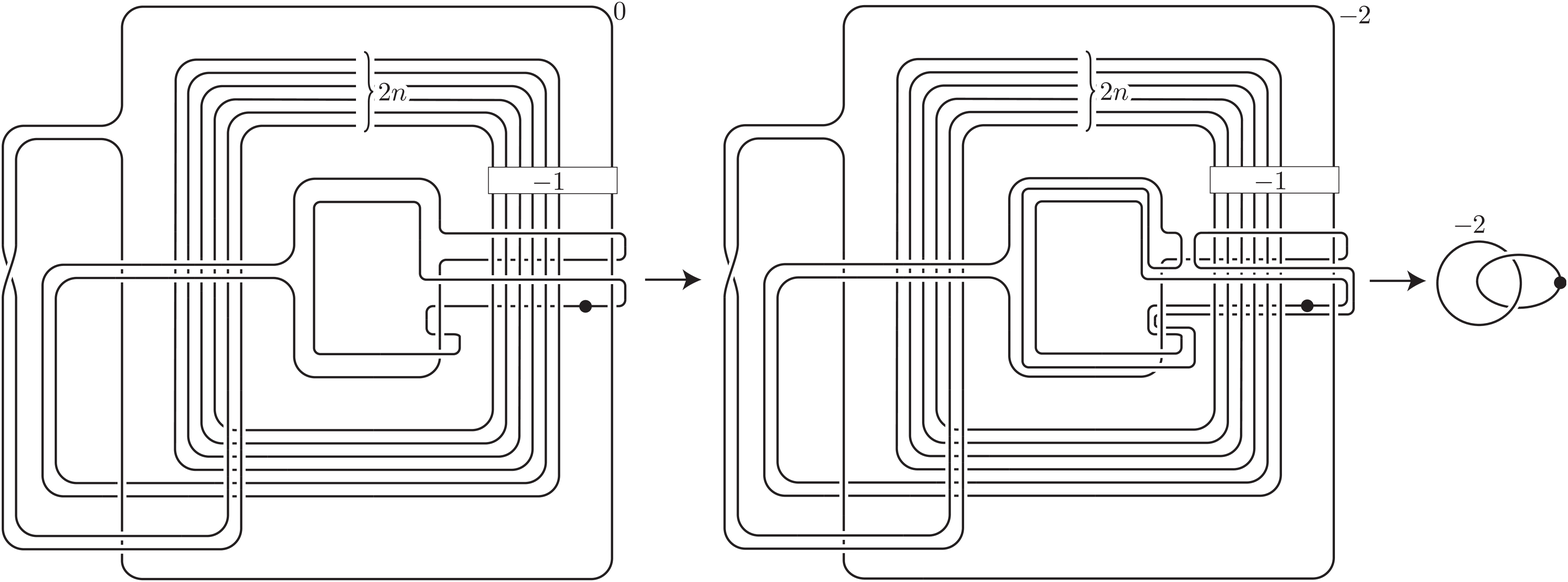}
\caption{$W \approx D^4$.}
\label{fig:4-manifold1}
\end{figure}
\begin{figure}[htb]
\includegraphics[width=50mm]{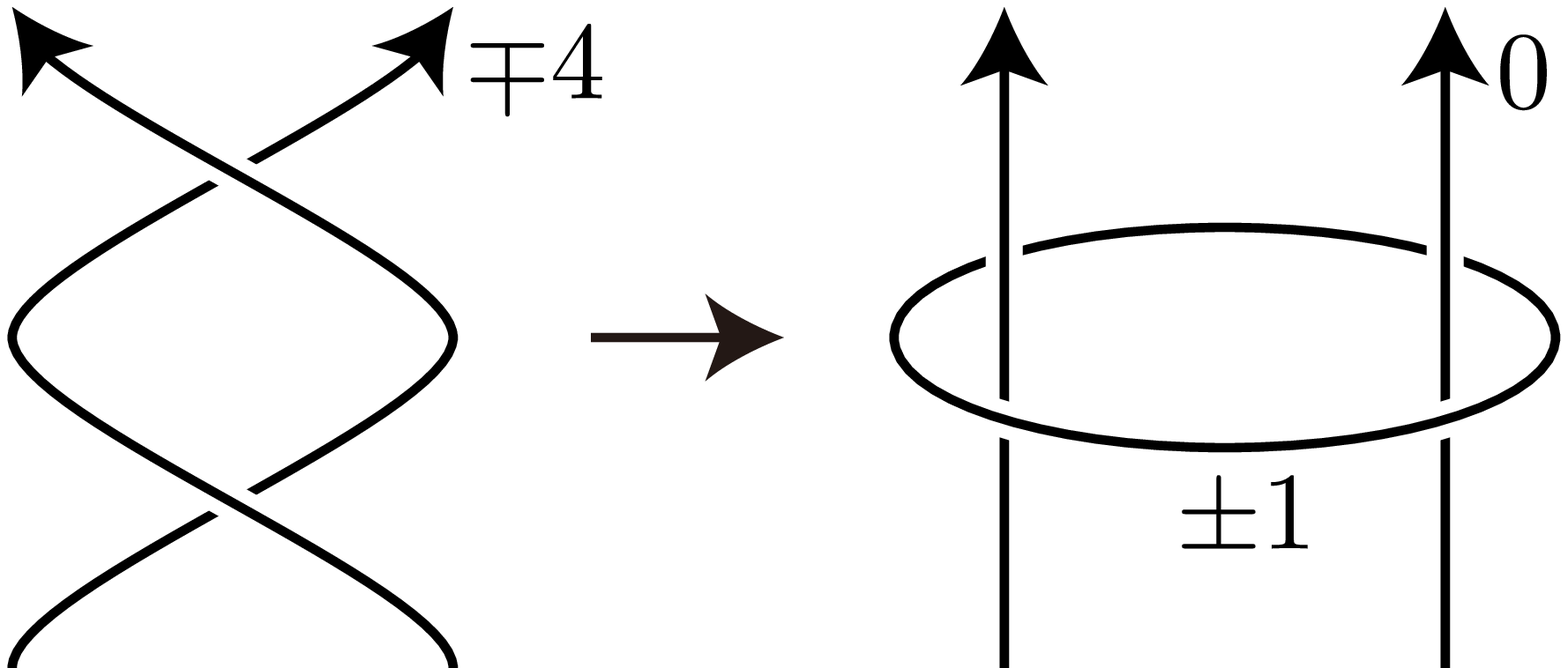}
\caption{}
\label{fig:blow-up}
\end{figure}

\medskip 
\noindent
\textbf{Case 3.} $\varepsilon =-1$  and $\gamma=0$.

\medskip 

In this case, the argument is similar to that for Case 1. 
The difference is that we consider a diffeomorphism from $\partial X_{K_{n-1}}(0)$ to $\partial X_{K_{n}}(0)$. 

First, we consider the case where $K=J_{1}$ with the band presentation $(A,b,c,-1)$ depicted in the left side of Figure~\ref{fig:Jm}.
Then the induced framing is $0$. 
Figure~\ref{fig:Kirby-Moves2} illustrates a diffeomorphism from $\partial X_{K_{n-1}}(0)$ to $\partial X_{K_{n}}(0)$ for an integer $n$. 
To see this, recall that $K_{n-1}=(K_{n})_{-1}$ as noted in Remark~\ref{rem:n-times}, and ignore the meridian $\mu'$ of $K_{n-1}$  for a while, which is illustrated by a broken circle in Figure~\ref{fig:Kirby-Moves2}. 
The first Kirby diagram in Figure~\ref{fig:Kirby-Moves2} is obtained from $X_{K_{n-1}}(0)$ by a blow-up. 
Sliding the $1$-framed unknot to the $0$-framed $K_{n}$, we obtain the second Kirby diagram.
Applying an annulus twist, we obtain the third Kirby diagram. 
After an ambient isotopy and a blow-down, we obtain the $0$-framed $K_{n}$ which represents $X_{K_{n}}(0)$. 
This Kirby calculus induces a diffeomorphism from $\partial X_{K_{n-1}}(0)$ to $\partial X_{K_{n}}(0)$ 
and we denote it by $f'_{n-1}$.

Remember the meridian $\mu'$ of $K_{n-1}$. 
We can check that $f'_{n-1}(\mu')$ is the $0$-framed unknot in the Kirby diagram of $X_{K_n}(0)$ if $\mu'$ is $0$-framed. 
Let $W'$ be the $4$-manifold $D^4 \cup h^1 \cup h^2$, 
where $h^1$ is the $1$-handle represented by $f'_{n-1}(\mu')$ with a dot and $h^2$ 
is the $2$-handle represented by $K_{n}$ with framing $0$.
Sliding $h^2$ over $h^1$, we obtain a canceling pair (see Figure~\ref{fig:4-manifold2}), implying that $W' \approx D^4$.
By Lemma~\ref{lem:extend}, $X_{K_{n-1}}(0) \approx X_{K_{n}}(0)$,
proving Theorem~\ref{thm:main} in the case where $K=J_{1}$ with the band presentation.

Furthermore we see that the above argument can be applied to any knot admitting a band presentation with $\varepsilon = -1$ and $\gamma = 0$ since all of a blow-up, a blow-down, handle slides, and isotopies in Figures~\ref{fig:Kirby-Moves2} and~\ref{fig:4-manifold2} were done in the neighborhood of the annulus $A$. 
Now we have completed the proof for Case 3.

\begin{figure}[htb]
\includegraphics[width=1.0\textwidth]{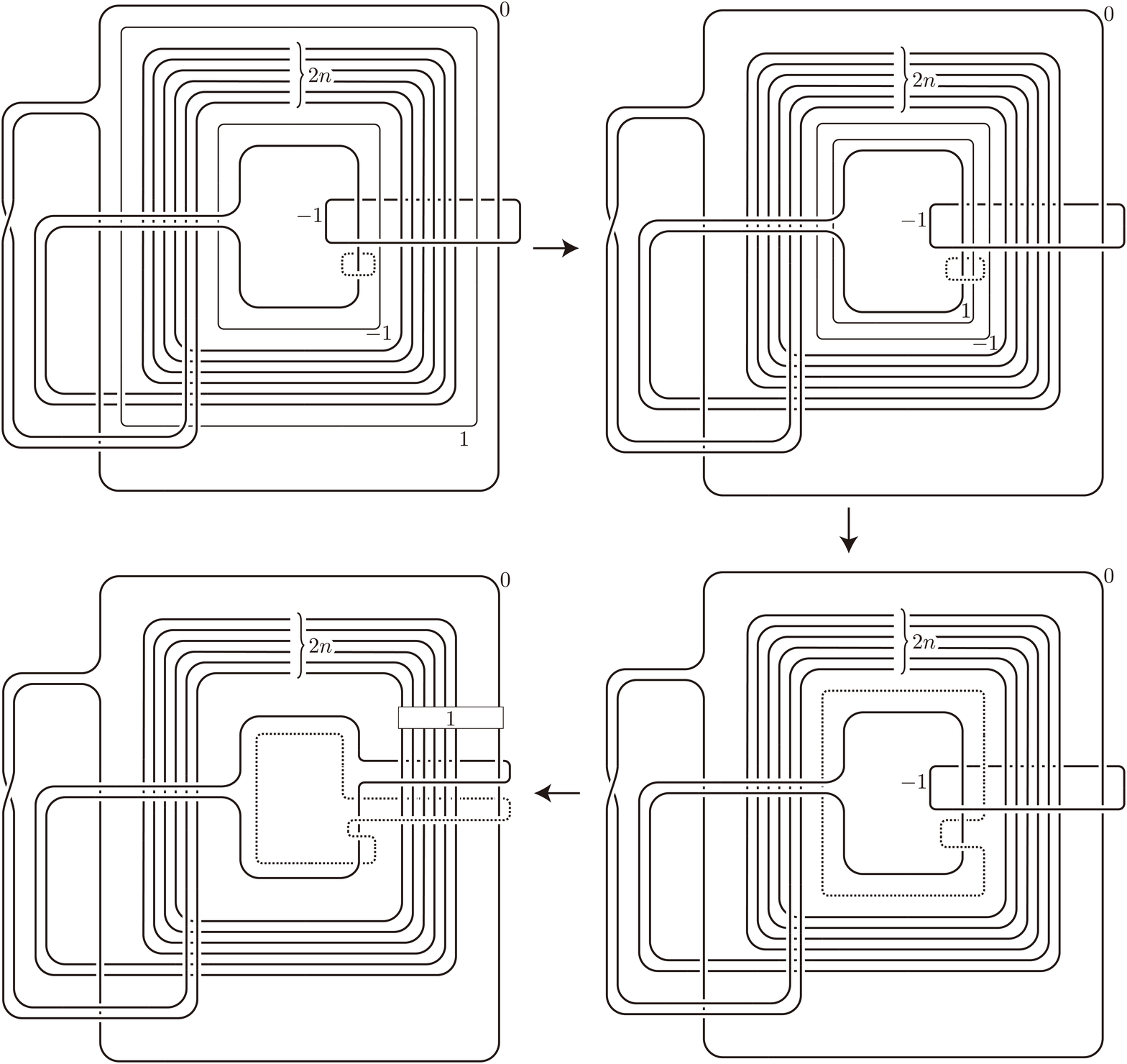}
\caption{Kirby diagrams representing a diffeomorphism $f_{n-1} : \partial X_{K_{n-1}}(0) \to \partial X_{K_{n}}(0) $.}
\label{fig:Kirby-Moves2}
\end{figure}
\begin{figure}[htb]
\includegraphics[width=80mm]{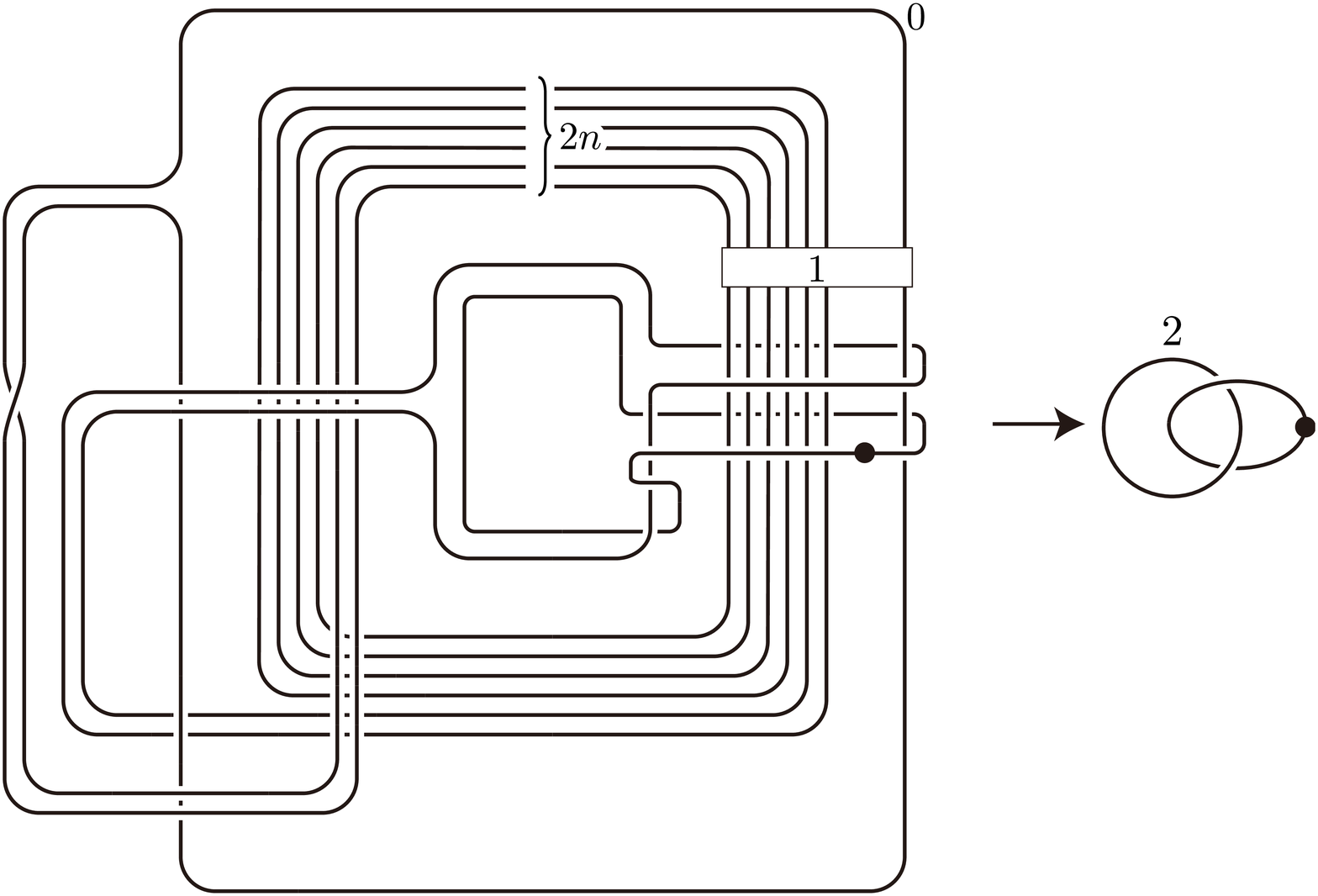}
\caption{$W' \approx D^4$.}
\label{fig:4-manifold2}
\end{figure}

\medskip 
\noindent
\textbf{Case 4.} $\varepsilon =-1$ and $\gamma=4$.
\medskip 

In this case, we can prove similarly as Case 3 since the framing $\gamma =4$ changed to $0$ when we apply a blow-up as shown in Figure~\ref{fig:blow-up}. 

\medskip 

Now we have completed the proof of Theorem~\ref{thm:main}. 
\end{proof}

Let $K$ be a knot with a band presentation $(A,b,c,\varepsilon)$,
$A'$ the associated annulus of $A$, and set $\partial A' = c_1 \cup c_2$.
The \textit{augmented $3$-component link} $L$ associated to the band presentation is the link 
$\left( \partial A \setminus b(\partial I \times I)\right) \cup b( I \times \partial I) \cup c_1 \cup c_2$ in $M_{\varepsilon}(c) \approx S^3$ and denote it by $L = K \cup c_1 \cup c_2$.

\begin{cor} \label{cor:problem}
For each $\gamma \in \{0, \pm 4\}$, there exist infinitely many mutually distinct knots $K_1, K_2, \cdots$
such that  $X_{K_{i}}(\gamma) \approx X_{K_{j}}(\gamma)$ for each $i, j \in \N$. 
\end{cor}
\begin{proof}
Let $J_{m}$ be the knot introduced in Example~\ref{ex:Jn}. 
First we prove Corollary~\ref{cor:problem} for $\gamma = \pm 4$. 
Consider the band presentation of $J_{m}$ depicted in the left (resp.~right) side of the lower half of Figure~\ref{fig:Jm}. 
Assume that $m$ is even (resp.~odd). 
Then the induced framing $\gamma$ is $4$ (resp.~$\gamma = -4$). 
Let $L = J_{m} \cup c_1 \cup c_2$ be the augmented 3-component link associated to the band presentation as shown in Figure~\ref{fig:AugLink}. 
Then $L$ admits the tangle decomposition as shown in Figure~\ref{fig:decomp1}. 
Since this decomposition is same as that in~\cite[Section 3]{Teragaito}, $L$ is hyperbolic by \cite[Proposition 3.2]{Teragaito}. 
Let $K_{n}$ be the knot obtained from $J_{m}$ by applying an annulus twist $n$ times. 
The knot complement $S^3 - K_{n}$ is obtained from  $S^3 - L$  by  $(1 + 1/n)$-filling on $c_1$ and $(1 -1/n)$-filling on $c_2$. 
By Thurston's hyperbolic Dehn surgery theorem, $S^3 - K_{n}$ is hyperbolic for any large enough $n$. 
Furthermore, the volume of $S^3 - K_{n}$ monotonically increases to the volume of  $S^3 - L$ for any large enough $n$ \cite{Neumann}. 
Since the volume is a knot invariant, %by Mostow rigidity, 
there exists a natural number $N$ such that $K_{i}$ and $K_{j}$ are mutually distinct for each $i, j \ge N$.
Finally, we redefine $K_{n}$ as the knot $K_{N+n}$. 
Then these knots are mutually distinct and $X_{K_{i}}(\gamma) \approx X_{K_{j}}(\gamma)$ for each $i, j \in \N$. 
Now we have completed the proof for the case where $\gamma = \pm 4$. 

Next we give the proof for the case where $\gamma = 0$. 
Consider the band presentation of $J_{m}$ depicted in the left side of the lower half of Figure~\ref{fig:Jm}, and assume that $m$ is odd. 
Then the induced framing $\gamma$ is $0$. 
In this case, the augmented 3-component link $L$ admits the tangle decomposition as shown in Figure~\ref{fig:decomp2}. 
The difference of Figures~\ref{fig:decomp1} and~\ref{fig:decomp2} is just glueing labels. 
In this case, we can easily check that the Conway polynomial of $L$ is non-zero. 
Thus $L$ is not a split link, and the exterior $E(L)$ is irreducible. 
We can also see that $E(L)$ is atoroidal and not Seifert fibered by the same argument in the proof of \cite[Proposition 3.2]{Teragaito}. 
Therefore $L$ is also hyperbolic in this case, and then the remaining proof is achieved by the same argument as the case where $\gamma = \pm 4$. 
Now we have completed the proof of Corollary~\ref{cor:problem}. 
\end{proof}
%%%%%%%%%%%%%%%%%%%%%%%%%%%%
\begin{figure}[htb]
\includegraphics[width=1.0\textwidth]{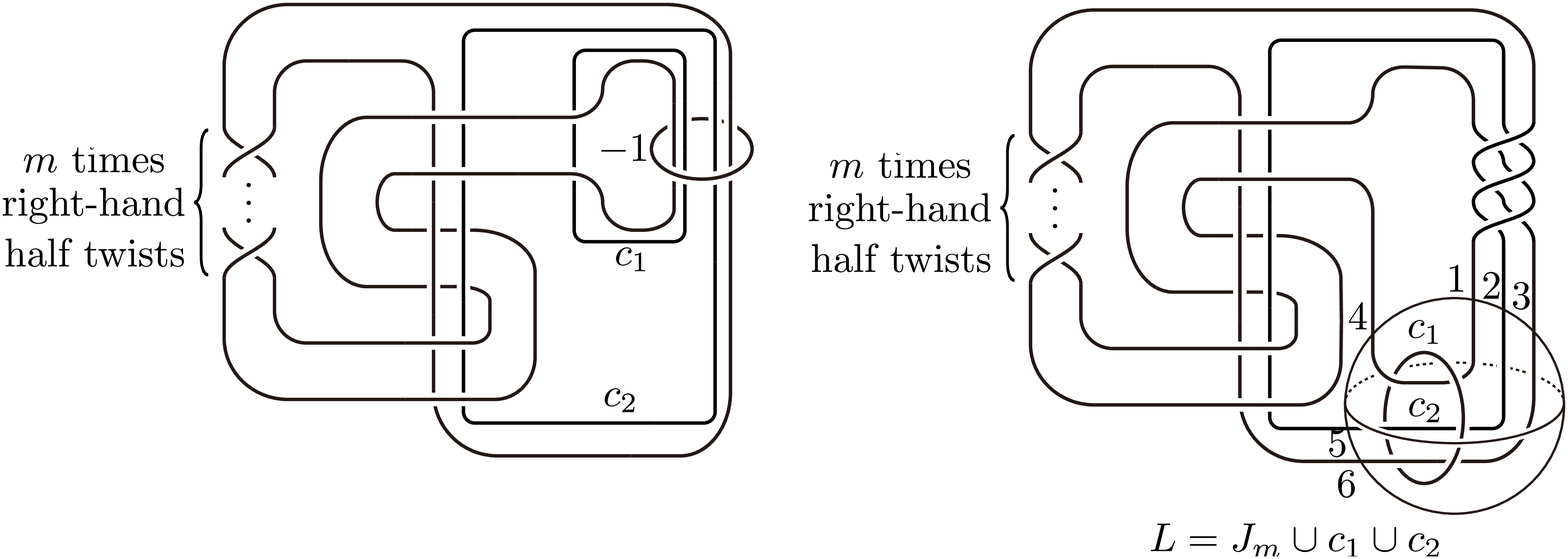}
\caption{The augmented 3-component link. }
\label{fig:AugLink}
\end{figure}
\begin{figure}[htb]
\includegraphics[width=.4\textwidth]{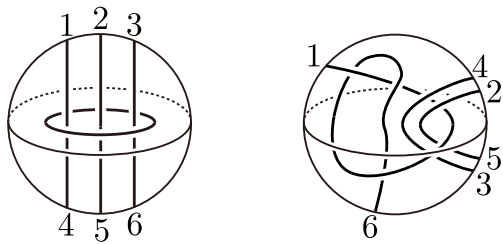}
\caption{Tangle decomposition in the case $\gamma = \pm 4$.}
\label{fig:decomp1}
\end{figure}
\begin{figure}[htb]
\includegraphics[width=.4\textwidth]{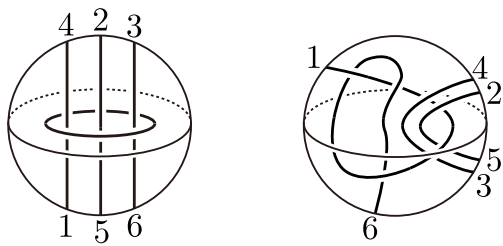}
\caption{Tangle decomposition in the case $\gamma = 0$.}
\label{fig:decomp2}
\end{figure}
%%%%%%%%%%%%%%%%%%%%
\begin{rem}
%With the notation of Example~\ref{ex:Jn},
%Theorem~\ref{thm:main} was proved by the fourth author \cite{Takeuchi1} for $K=J_2$
%and by  the third author \cite{Omae1} for $K=J_1$ with the band presentations depicted in the lower half of Figure~\ref{fig:Jm} respectively. 
%Note that our Kirby calculus in the proof of Theorem~\ref{thm:main} is different from theirs. 
Here recall that a given knot $J$ is fibered if and only if 
$M_{J}(0)$ is a surface bundle over $S^1$~\cite[Corollary 8.19]{Gabai}. 
Since $J_{1}$ is fibered,
the knots obtained from $J_{1}$ by applying annulus twists are fibered
and these genera are two.
In this case, the genus does not distinguish these knots.
On the other hand, though $J_{2}$ is fibered, 
it is not obvious whether  the knot $K_{n}$ obtained from $J_{2}$  by applying an annulus twist $n$ times is fibered.
In~\cite{Takeuchi1}, the fourth author proved that $K_n$ $(n \in \N)$ is a fibered knot by studying the complementaly sutured manifold of $K_n$, and showed that the genus of $K_n$ is $2n+2$.
Therefore the genus distinguishes the knots $K_n$. 
%This implies that, to distinguish $K_{i}$ and $K_j$, it is enough to calculate the Alexander polynomial of $\Delta_{K_n}(t)$. 
\end{rem}

\section{The $n$-shake genus and the $4$-ball genus}\label{sec:SakeGenus}

In this section, we study a relationship between the $n$-shake genus and the 4-ball genus of a knot.

\subsection{The shake genus}
In this subsection, 
we focus on the $0$-shake genus. 
We simply call the $0$-shake genus the {\it shake genus} of a knot $K$ and we denote it by $g_{s}(K)$.
Akbulut and Kirby \cite[Problem~1.41 (A)]{Kirby} asked whether $g_s(K) = g_{*}(K)$ 
and noted that these probably do not coincide. 
One of our motivations is to find a knot $K$ with $g_s(K) < g_{*}(K)$. 
The following is a corollary of Theorem~\ref{thm:main}.

\begin{cor}\label{cor:ShakeGenus} 
Let $K$ be a knot with a band presentation whose induced framing is $0$, 
$K_{n}$ the knot obtained from $K$ by applying an annulus twist $n$ times. 
Then, for each $i, j \in \mathbb{Z}$, 
$$g_{s}(K_{i})=g_{s}({K_{j}}) \, .$$ 
\end{cor}

%Corollary~\ref{cor:ShakeGenus} would provide a candidate of a knot $K$ with $g_s(K) < g_{*}(K)$. 
%However, 
%it is not easy to show that our candidate satisfies the inequality. 
In the following, 
we observe a property of knots constructed by annulus twists
(Proposition~\ref{prop:un}). 
Recall that a knot $K \subset S^3$ is said to be {\it slice} if $g_{*}(K)=0$. 
We say that a knot $K \subset S^3$ is {\it homotopically slice} if there exists a homotopy $4$-ball $W$ with $\partial W \approx S^3$ such that $K$ bounds a smoothly embedded disk in $W$.

\begin{lem} \label{lem:homotopy-4ball}
Let $K$ and $K'$ be knots such that $M_{K}(0)  \approx M_{K'}(0)$.
If $K$ is slice, then $K'$ is homotopically slice. 
\end{lem}
\begin{proof}
Since $K$ is slice, 
there exists a  properly embedded disk  $D^2$ in $D^4$ such that $\partial D^2=K$
and let  $\nu D^2$ be an open tubular neighborhood of $D^2$ in $D^4$. 
Notice that $\partial (D^4 \setminus \nu D^2) \approx M_{K}(0)$.
Let $k$ be a diffeomorphism from $M_{K}(0)$ to $M_{K'}(0) $ and  we regard $\partial (D^4 \setminus \nu D^2) $ as $M_{K'}(0)$ via $k$.
In other words,  we consider the $4$-manifold 
\[ (M_{K'}(0) \times [-1, 1])  \cup _{k} (D^4 \setminus \nu D^2),\]
where $\partial (D^4 \setminus \nu D^2) \approx M_{K}(0)$ and
$M_{K'}(0) \times \{-1\} \approx M_{K'}(0)$ are identified by $k$. 
Taking a meridian $\mu'$ of $K'$, add a $2$-handle along 
$\mu' \subset M_{K'}(0) \approx  M_{K'}(0) \times \{1\} = \partial ((M_{K'}(0) \times [-1, 1])  \cup _{k} (D^4 \setminus \nu D^2))$
with the framing $0$. 
We denote by $W$ the resulting $4$-manifold. 
Then we see that $\partial W  \approx S^3$, 
and $K'$ is isotopic to the boundary of the cocore disk of the $2$-handle attached along $\mu'$. 
Thus $K'$ bounds the cocore disk in $W$, that is, a smoothly embedded disk in $W$. 

We show that $W$ is a homotopy 4-ball, 
i.e., $\pi_{*}(W) \simeq \pi_{*}(D^4)$.
To prove this, we show that $W$ is a homology 4-ball, i.e., $H_{*}(W) \simeq H_{*}(D^4)$, and $\pi _{1}(W)$ is trivial. 
First we prove that $H_{*}(W) \simeq H_{*}(D^4)$. 
We consider the inclusion map
\[ i:  (S^3 \setminus \nu K') \to (M_{K'}(0) \times [-1, 1])  \cup _{k} (D^4 \setminus \nu D^2), \]
 where $\nu K'$ is an open tubular  neighborhood of $K'$ in $S^3$.
Then an elementary homological argument tells us that  $i_{*}$ is an isomorphism.
Since  $H_{1}(S^3 \setminus \nu K')$ is generated by $\mu'$, 
$H_{1}((M_{K'}(0) \times [-1, 1])  \cup _{k} (D^4 \setminus \nu D^2))$ is also generated by $\mu'$. 
Adding the $2$-handle along $\mu'$ kills $H_{1}\left ( \left( M_{K'}(0) \times [-1, 1] \right)  \cup _{k} \left( D^4 \setminus \nu D^2 \right) \right)$ and 
therefore $W$ is a homology 4-ball. 

Next, we prove that $\pi _{1}(W)$ is trivial. 
By considering 
a handle decomposition of $D^4 \setminus \nu D^2$,
we obtain that $\pi _{1} (D^4 \setminus \nu D^2)$ is normally generated by a meridian $\mu$
of $K$.
Therefore 
 $\pi _{1} \left( \left( M_{K'}(0) \times [-1, 1] \right)  \cup _{k} \left( D^4 \setminus \nu D^2 \right) \right)$ is normally generated by $\mu '$. 
By the van Kampen theorem, $\pi _{1}(W)$ is trivial.
Since $H_{*}(W) \simeq H_{*}(D^4)$, $W$ is a homotopy 4-ball. 
\end{proof}

By Proposition \ref{prop:3-diffeo} and Lemma~\ref{lem:homotopy-4ball}, 
we have the following.

\begin{prop} \label{prop:un}
Let $K$ be a slice knot admitting 
a band presentation whose induced framing is $0$ and 
$K_{n}$ the knot obtained from $K$ by applying an annulus twist $n$ times.
Then there exists a homotopy $4$-ball $W_{n}$ with $\partial W_{n}=S^3$
such that $K_{n}$ bounds a smoothly embedded disk in $W_{n}$.
In particular, we can associate a homotopy $4$-sphere for each $n \in \Z$.
\end{prop}

By Lemma~\ref{lem:band-presentation}, an unknotting number one knot admits a band presentation. 
Furthermore we may assume that the induced framing of the band presentation is $0$ as in Example~\ref{ex:Jn}. 
Therefore we have the following. 

\begin{prop} \label{prop:un2}
We can obtain homotopically slice knots and homotopy $4$-spheres from a slice knot with unknotting number one.
\end{prop}

Let $K$ be a knot with a band presentation such that the induced framing is $0$, 
and $K'$ the knot obtained from $K$ by a single annulus twist. 
Then $g_{s}(K)=g_s(K')$ by Corollary~\ref{cor:ShakeGenus}.
If  $g_{*}(K) \neq g_*(K')$, then we obtain an answer to Akbulut and Kirby's problem.
However, it  seems that $g_{*}(K)=g_*(K')$ by the following observation. 

First, since $K$ can be transformed into the unknot by two band surgeries, 
we see that  $g_{*}(K) \le 1$.
Suppose that  $K$ is slice.
Then $K'$ is homotopically slice by Lemma~\ref{lem:homotopy-4ball}. 
%Therefore slicing obstructions such as the Rasmussen invariant~\cite{Rasmussen} or the $\tau$-invariant~\cite{OzsbathSzabo} vanish. 
If the smooth Poincar\'e conjecture in dimension four is true, 
then $K'$ is  slice, and then $g_*(K) = g_*(K')=0$. 
Suppose that $g_*(K) = 1$.
Then it is likely that $g_*(K') = 1$ since if $K'$ is slice then
$K$ is homotopically slice by Lemma~\ref{lem:homotopy-4ball},
which produce a counterexample for  
the smooth Poincar\'e conjecture in dimension four.
Consequently, if the smooth Poincar\'e conjecture in dimension four is true, then $g_*(K) = g_*(K')$. 
Based on the above observations we propose the following. 

\begin{conj}\label{conj:AJ}
Let $K$ be a knot with a band presentation whose induced framing is $0$, 
and $K_{n}$ the knot obtained from $K$ by applying an annulus twists $n$ times. 
Then $g_{*}(K)=g_{*}({K_{n}})$.  
\end{conj}

\subsection{The $n$-shake genus for $n \ne 0$}

In this subsection, we consider the $n$-shake genus of a knot for $n \ne 0$. 
Let $K_{n,m}$ be the knot depicted in the left side of Figure~\ref{fig:OmaeAndRm}, 
where the shaded rectangle labeled $2m+1$ represents $2m+1$ times left-hand half twists as shown in Figure~\ref{fig:HalfTwists}. 
Note that   $K_{n,m}$ is a generalization of  knots in \cite{Akbulut2, Akbulut1},  \cite{Lickorish}, and \cite{Omae1}.
Then we have the following. 

\begin{figure}[htb]
\includegraphics[width=.9\textwidth]{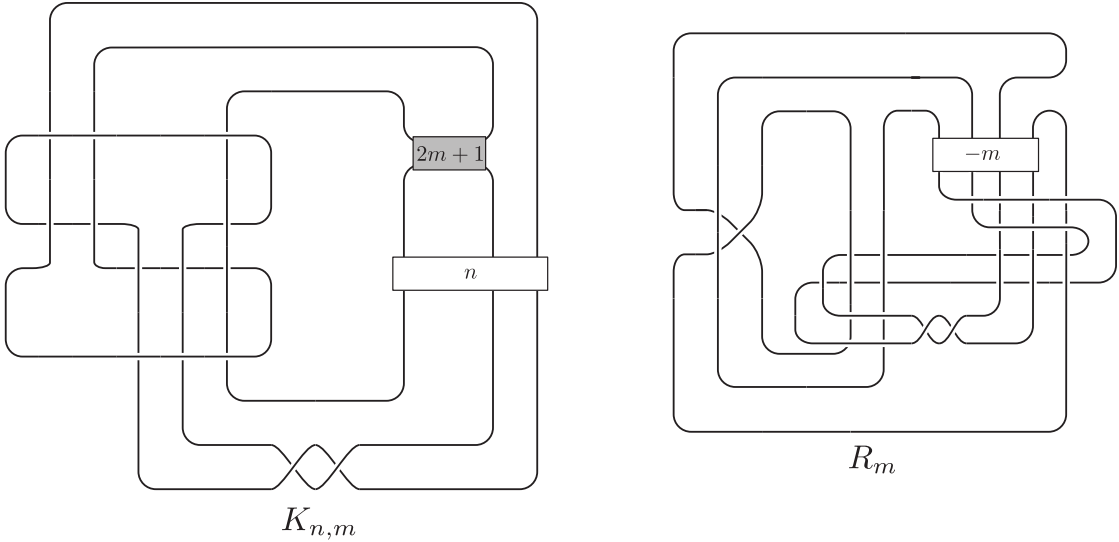}
\caption{The knots $K_{n,m}$ and $R_m$. }\label{fig:OmaeAndRm}
\end{figure}
\begin{figure}[htb]
\includegraphics[width=.25\textwidth]{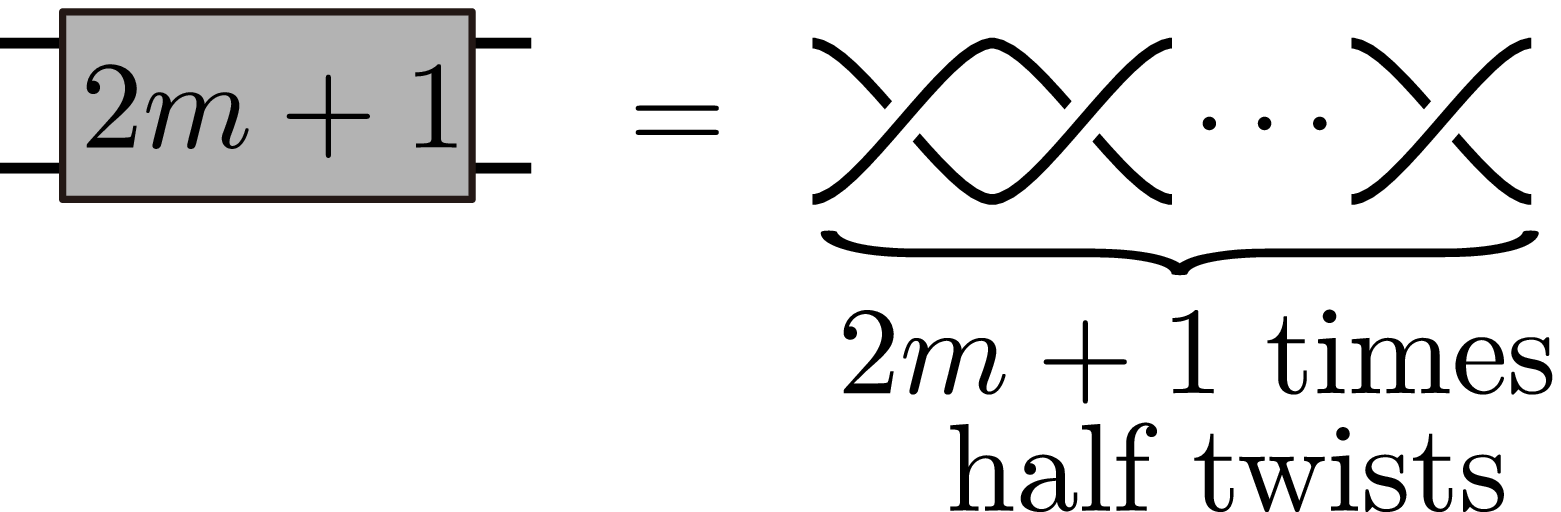}
\caption{}\label{fig:HalfTwists}
\end{figure}

\begin{thm} \label{thm:Omae}
For integers $n \neq 0$ and $m \ge 0$,  
\[ g_{s}^{n}(K_{n,m} )=0\ \text{and}\ g_{*}(K_{n,m})=1 \, . \]
\end{thm}

We obtain the following immediately as a corollary of Theorem~\ref{thm:Omae}. 

\begin{cor}\label{cor:nShakeGenus}
For each integer  $n \neq 0$, 
there exist infinitely many knots $K_1, K_2, \dots$ such that 
$ g_{s}^{n}(K_i) < g_{*}(K_i)$ for any $i \in \N$. 
\end{cor}

To prove Theorem~\ref{thm:Omae}, we give two lemmas. 
We normalize the Alexander polynomial $\Delta_K(t)$ of a knot $K$ so that 
$\Delta_K(t^{-1}) = \Delta_K(t)$ and $\Delta_K(1) = 1$.

\begin{lem}\label{lem:Alexander}
Let $n$ be a positive integer. Then for any $m \ge 0$, 
\begin{align*}
\Delta_{K_{n,m}}(t) =
& -(1+6m) + (2+4m)(t + t^{-1})-(1+m)(t^2 + t^{-2}) +m (t^{n-2} + t^{-(n-2)}) \\ 
& -(1+3m) (t^{n-1} + t^{-(n-1)}) +(2+3m)(t^n + t^{-n})-(1+m) (t^{n+1} + t^{-(n+1)})
\end{align*}
and 
\begin{align*}
\Delta_{K_{-n,m}(t)}= 
&-(1+6m) + (2+4m)(t + t^{-1}) -(1+m)(t^2 + t^{-2}) -(1+m) (t^{n-1} + t^{-(n-1)}) \\
& +(2+3m)(t^n + t^{-n}) -(1+3m) (t^{n+1} + t^{-(n+1)}) +m (t^{n+2} + t^{-(n+2)}) \, .
\end{align*}
\end{lem}

The proof of Lemma~\ref{lem:Alexander} is achieved by a direct calculation of the Alexander polynomial 
by using a Seifert matrix.  
We omit the proof of Lemma~\ref{lem:Alexander}. 

\begin{lem}\label{lem:SliceObst} 
Let $K$ be a slice knot. 
Then the constant term of $\Delta_K(t)$ is positive. 
\end{lem}
\begin{proof}%[Proof of Lemma~\ref{lem:SliceObst}]
%Let $\Delta_{K}(t)= a_{0}+ a_{1}(t + t^{-1})+a_{2}(t^2 + t^{-2})+ \cdots + a_l(t^l + t^{-l})$. 
Since $K$ is slice, there exists a polynomial $F(t) = \sum\limits_{i \ge 0} b_i t^i$ such that $\Delta_{K}(t) = F(t) F(t^{-1})$. 
This fact was proved by Fox and Milnor~\cite{FoxMilnor}, and independently by Terasaka~\cite{Terasaka}.  
Then we see that the constant term of $\Delta_K(t)$ is equal to $\sum\limits_{i \ge 0} b_{i}^2$. 
\end{proof}
Recall that a knot $K \subset S^3$ is said to be {\it ribbon} if $K$ is the boundary of a 
smoothly immersed disk $D^2 \looparrowright S^3$ with only ribbon singularities.
Clearly a ribbon knot is slice. 
Now we prove Theorem~\ref{thm:Omae}.

\begin{proof}[Proof of Theorem~\ref{thm:Omae}]
First, we determine the $n$-shake genus $g_s^n(K_{n,m})$.
Let $R(m)$ be the knot depicted in the right side of Figure~\ref{fig:OmaeAndRm}. 
Then we have the following. 

\begin{clm} \label{clm:diffeo}
$X_{K_{n,m}}(n) \approx X_{R(m)}(n)$.
In particular, $g_{s}^{n}(K_{n,m} )=0$.
\end{clm}
\begin{proof}
We see that $X_{K_{n,m}}(n) \approx X_{R(m)}(n)$ by Kirby calculus as shown in Figures~\ref{fig:OmaeKnot1} and \ref{fig:OmaeKnot2}. 
Thus we have $g_{s}^{n}(X_{K_{n,m}}(n)) =g_{s}^{n}(X_{R(m)}(n))$. 
Since $R(m)$ is ribbon, $g_{s}^{n}(R(m))=g_{*}(R(m))=0$.
Therefore $g_{s}^{n}(X_{K_{n,m}}(n)) =0$.
\end{proof}

\begin{figure}[htb]
\includegraphics[width=.95\textwidth]{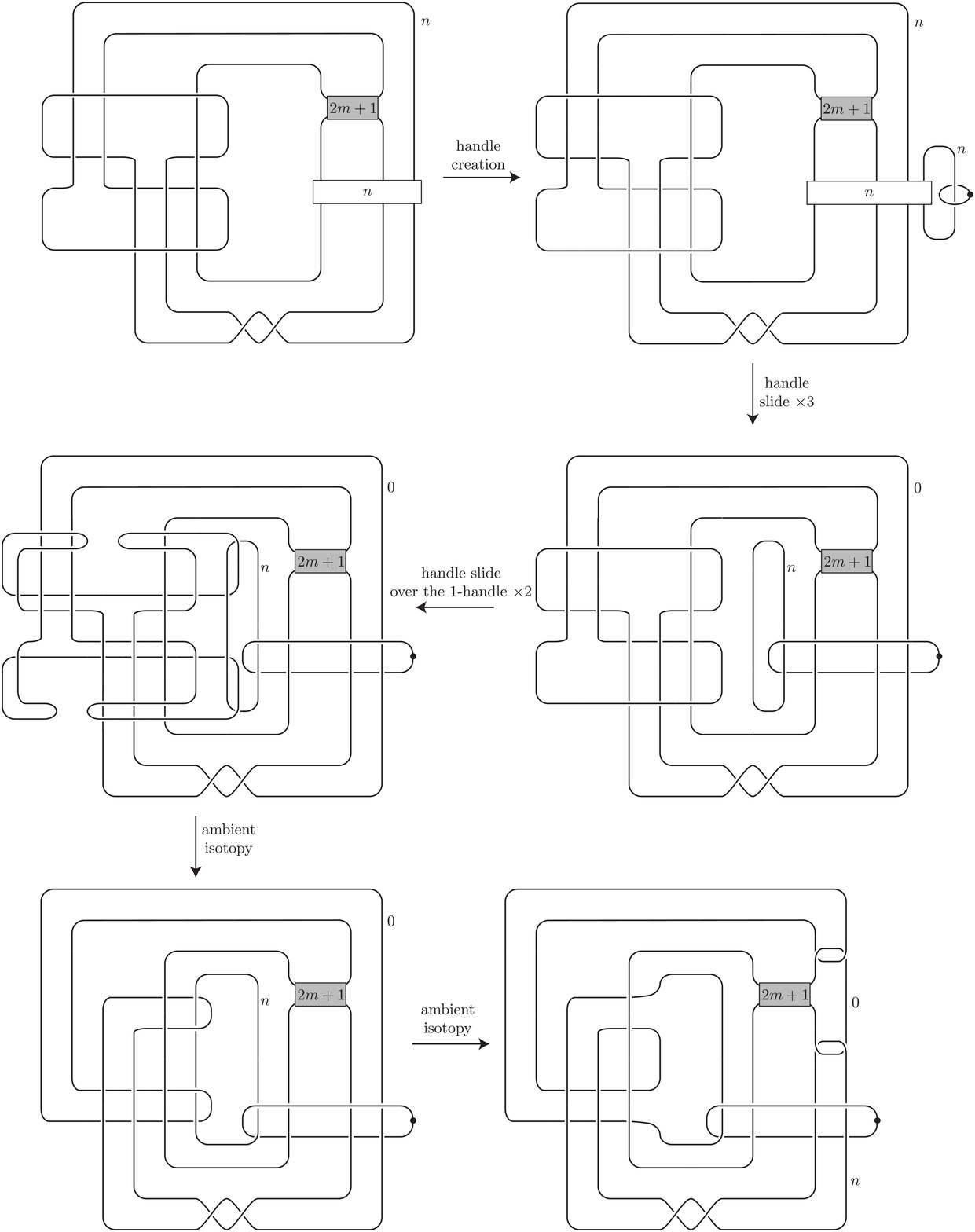}
\caption{$X_{K_{n,m}}(n) \approx X_{R(m)}(n)$. }\label{fig:OmaeKnot1}
\end{figure}
\begin{figure}[htb]
\includegraphics[width=.95\textwidth]{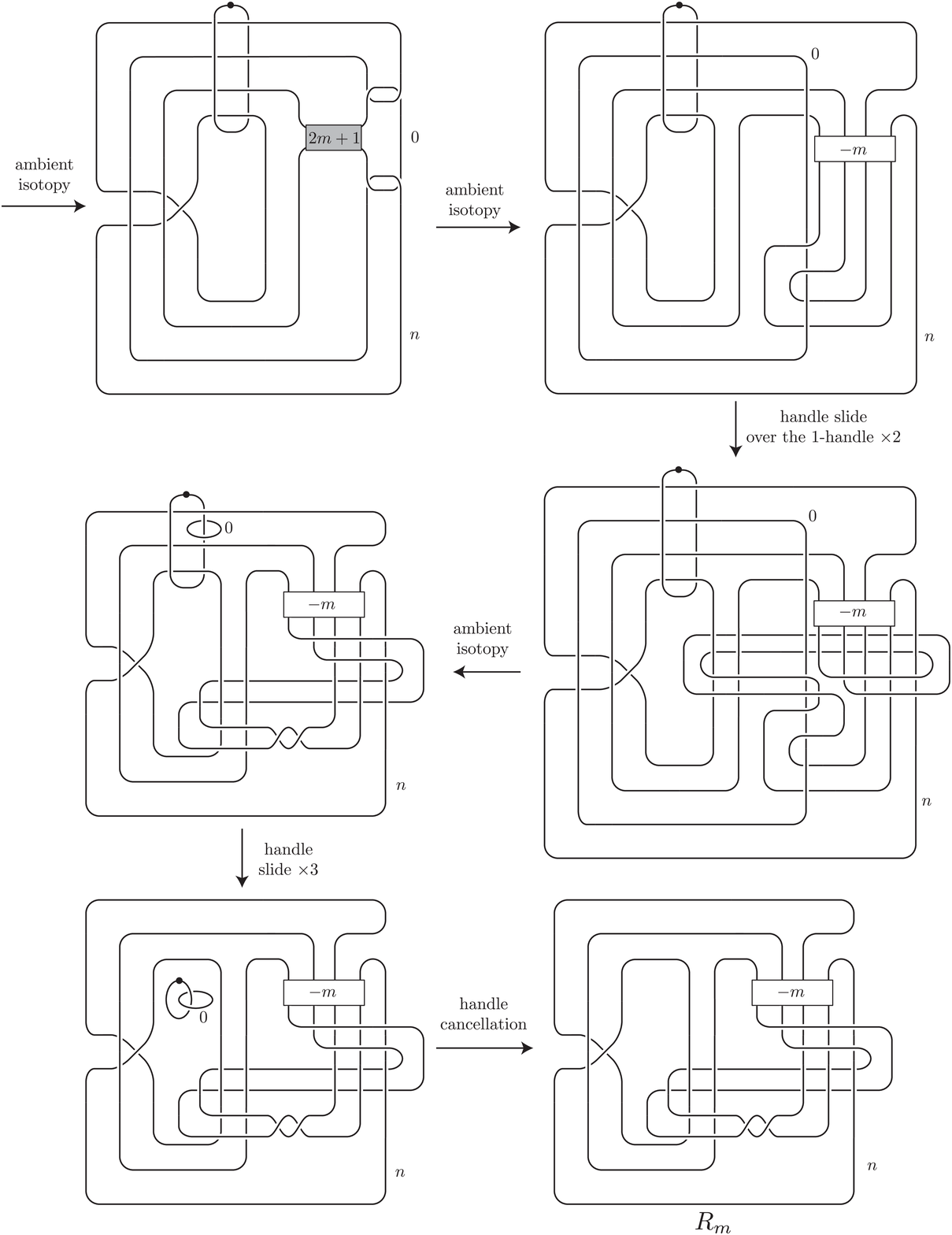}
\caption{$X_{K_{n,m}}(n) \approx X_{R(m)}(n)$. (continued)}\label{fig:OmaeKnot2}
\end{figure}
%\begin{figure}[htb]
%\includegraphics[width=1.0\textwidth]{OmaeKnot_3.eps}
%\caption{$X_{K_{n,m}}(n) \approx X_{R(m)}(n)$. (continued)}\label{fig:OmaeKnot3}
%\end{figure}

Next we determine the $4$-ball genus $g_*(K_{n,m})$ for $n \ne 0$ and $m \ge 0$.

\begin{clm}\label{clm:4-genus}
For integers $n \ne 0$ and $m \ge 0$, 
$g_{*}(K_{0,m})=0 \text{ and } g_{*}(K_{n,m} )=1$. 
\end{clm}
\begin{proof}
We see that $K_{0,m}$ is ribbon. 
Therefore $g_{*}(K_{0,m})=0$. 
Since $\Delta_{K_{n,m}}(0) < 0$ by Lemma~\ref{lem:Alexander}, 
we see that $g_*(K_{n,m}) \ge 1$ by Lemma~\ref{lem:SliceObst}. 
Since $K_{n,m}$ can be transformed into the unknot by two band surgeries, 
we see that  $g_{*}(K_{n,m}) \le 1$. 
Thus we have $g_*(K_{n,m}) = 1$ for $n \ne 0$ and $m \ge 0$. 
%Now we have completed the proof of Claim~\ref{clm:4-genus}
\end{proof}
Now we have completed the proof of Theorem~\ref{thm:Omae}. 
\end{proof}

\begin{rem}\label{rem:Sakuma}
We can check that $\Delta_{R(m)}(t) =3 - (t^2 +t^{-2})$.
Let $\omega$ be $n$-th root of the unity. 
Then $\Delta_{K_{n,m}}(\omega) = \Delta_{R(m)}(\omega)=3-(\omega ^2 + \omega^{-2})$. 
In fact, let $K_1$ and $K_2$ be knots such that $\partial X_{K_1}(n) \approx \partial X_{K_2}(n)$.
Then $H_{1} (\Sigma_{n}(K)) \cong H_{1} (\Sigma_{n}(K'))$, where $\Sigma_n(K)$ denotes the $n$-fold cyclic branched covering space of $S^3$ branched along $K$. 
In particular, $\Delta_{K}(\omega) = \Delta_{K'}(\omega)$. 
%This fact was suggested by M.~Sakuma.
\end{rem}

%%%%%%%%%%%%%%%%%%%%%%%%%%%%%%%%%%%%%%%%%%%%%%%%%%%%%%%%%%%%%%%

\begin{thebibliography}{9}


\bibitem{Abe1}
T.~Abe, R.~Hanaki and R.~Higa.
\text{The unknotting number and band-unknotting number of a knot},
\textit{Osaka J. Math} \textbf{49} (2012), no.~2, 523--550.
  
\bibitem{Abe2}
T.~Abe and T.~Kanenobu.
 \text{Unoriented band-surgery on knots and links},
\textit{preprint}. 
 
\bibitem{Akbulut2}
S.~Akbulut.
 \text{Knots and exotic smooth structures on $4$-manifolds},
\textit{J. Knot Theory Ramifications} \textbf{2} (1993), no. 1, 1--10. 

\bibitem{Akbulut1}
S.~Akbulut. 
\text{On $2$-dimensional homology classes of $4$-manifolds},
\textit{Math. Proc. Cambridge Philos. Soc}. \textbf{82} (1977), no. 1, 99--106.
 
\bibitem{AkbulutBook}
S.~Akbulut. 
 \textit{$4$-manifolds},
draft of a book (2012), \\ 
available at {\tt http://www.math.msu.edu/\~{}akbulut/papers/akbulut.lec.pdf}
 
\bibitem{Brakes}
R.~Brakes.
\text{Manifolds with multiple knot-surgery descriptions},
 \textit{Math. Proc. Cambridge Philos. Soc}. \textbf{87} (1980), no. 3, 443--448. 
  
  
\bibitem{Cerf}
J.~Cerf.
\text{Sur les diffeomorphismes de la sphere de dimension trois $(\Gamma_{4}=0)$},
Lecture Notes in Mathematics, No. 53, Springer-Verlag, Berlin-New York (1968) xii+133 pp. 


\bibitem{FoxMilnor}
R.~H.~Fox and J.~W.~Milnor.
 \text{Singularities of $2$-spheres in $4$-space and cobordism of knots}, 
\textit{Osaka J. Math}. \textbf{3} (1996), 257--267.

\bibitem{Gabai}
D.~Gabai.
\text{Foliations and the topology of $3$-manifolds. III},
\textit{J. Differential Geom}. \textbf{26} (1987), no. 3, 479--536. 

\bibitem{Gompf2}
R.~Gompf  and K.~Miyazaki. 
\text{Some well-disguised ribbon knots},
\textit{Topology Appl}. \textbf{64} (1995), no. 2, 117--131. 

\bibitem{Gompf}
R.~Gompf and A.~Stipsicz.
\textit{$4$-manifolds and Kirby calculus}, 
Graduate Studies in Mathematics, \textbf{20}. American Mathematical Society, Providence, RI, 1999. xvi+558.

\bibitem{Kawauchi}
A.~Kawauchi.
\text{Mutative hyperbolic homology $3$-spheres with the same Floer homology},
\textit{Geom. Dedicata} \textbf{61} (1996), no. 2, 205--217.

\bibitem{Kirby}
R.~Kirby.
\text{{Problems in Low-Dimensional Topology}},
AMS/IP Stud. Adv. Math., {\bf 2}(2), Geometric topology (Athens, GA, 1993), 
35--473, Amer. Math. Soc., Providence, RI, 1997.

\bibitem{Kouno}
R.~Kouno. 
\text{$3$--manifold with infinitely many knot surgery descriptions}, 
(in Japanese) master thesis of Nihon University (2002). 

\bibitem{Neumann}
W.~Neumann and D.~Zagier.
\text{Volumes of hyperbolic three-manifolds},
\textit{Topology} \textbf{24} (1985), no. 3, 307--332. 


\bibitem{Lickorish}
W.~B.~R.~Lickorish.
\text{Shake-slice knots},
Lecture Notes in Math., \textbf{722} (1979), 67--70. 


\bibitem{Lickorish2}
W.~B.~R.~Lickorish.
\text{Surgery on knots},
 \textit{Proc. Amer. Math. Soc}. \textbf{60} (1976), 296--298.  


\bibitem{Livingston}
C.~Livingston.
\text{More $3$-manifolds with multiple knot-surgery and branched-cover descriptions},
 \textit{Math. Proc. Cambridge Philos. Soc}. \textbf{91} (1982), no. 3, 473--475.

%Topology Appl. \textbf{64} (1995), no. 2, 117--131,
 
\bibitem{Omae1}
Y.~Omae.
 \text{$4$-manifolds constructed from a knot and the shake genus}, 
 (in Japanese) master thesis of Osaka University (2011).


\bibitem{Osoinach}
J.~Osoinach.
\text{Manifolds obtained by surgery on an infinite number of knots in $S^3$}, 
\textit{Topology} \textbf{45} (2006), no. 4, 725--733.


\bibitem{Teragaito2}
T.~Saito and M.~Teragaito.
 \text{Knots yielding diffeomorphic lens spaces by Dehn surgery}, 
\textit{Pacific J. Math}. \textbf{244} (2010), no. 1, 169--192. 

\bibitem{Takeuchi1}
M.~Takeuchi.
 \text{Infinitely many distinct framed knots which represent  a diffeomorphic 4-manifold},
(in Japanese) master thesis of Osaka University (2009).

\bibitem{Teragaito}
M.~Teragaito.
 \text{A Seifert fibered manifold with infinitely many knot-surgery descriptions}, 
\textit{Int. Math. Res. Not. IMRN} 2007, no. 9, Art. ID rnm 028, 16 pp.

\bibitem{Teragaito3}
M.~Teragaito.
\text{Homology handles with multiple knot-surgery descriptions},
\textit{Topology Appl}. \textbf{56} (1994), no. 3, 249--257.

\bibitem{Terasaka}
H.~Terasaka.
 \text{On null-equivalent knots}, 
\textit{Osaka Math. J.} \textbf{11} (1959), 95--113. 

\bibitem{Winter}
B.~Winter.
 \text{On Codimension Two Ribbon Embeddings}, 
\textit{arXiv}:0904.0684 (2009). 

\end{thebibliography}
\end{document}